\documentclass[sn-mathphys-num]{sn-jnl}% Math and Physical Sciences Numbered Reference Style 
\usepackage{graphicx}%
\usepackage{multirow}%
\usepackage{amsmath,amssymb,amsfonts}%
\usepackage{amsthm}%
\usepackage{mathrsfs}%
\usepackage[title]{appendix}%
\usepackage{xcolor}%
\usepackage{textcomp}%
\usepackage{manyfoot}%
\usepackage{booktabs}%
\usepackage{algorithm}%
\usepackage{algorithmicx}%
\usepackage{algpseudocode}%
\usepackage{listings}%
\usepackage{caption}    % For customizing captions
\usepackage{subcaption} % For subtables
\expandafter\def\csname ver@subfig.sty\endcsname{}

\usepackage{float}      % For the 'H' placement option
\usepackage{booktabs,makecell}
\usepackage{adjustbox}
\usepackage{diagbox}

\usepackage{xcolor}

\usepackage{graphicx,epstopdf} % <- Preamble

\usepackage{array}

\usepackage{geometry}
\theoremstyle{thmstyleone}%
\newtheorem{theorem}{Theorem}%  meant for continuous numbers
\newtheorem{proposition}[theorem]{Proposition}% 

\theoremstyle{thmstyletwo}%

\theoremstyle{thmstylethree}%
\newtheorem{definition}{Definition}%

\begin{document}

\title[Article Title]{Vector Extrapolation Methods Applied To Geometric Multigrid Solvers For Isogeometric Analysis}


\author*[1,2]{\fnm{Abdellatif} \sur{Mouhssine}}\email{abdellatif.mouhssine@um6p.ma}

\author[1]{\fnm{Ahmed} \sur{Ratnani}}\email{ahmed.ratnani@um6p.ma}
\equalcont{These authors contributed equally to this work.}

\author[2]{\fnm{Hassane} \sur{Sadok}}\email{hassane.sadok@univ-littoral.fr}
\equalcont{These authors contributed equally to this work.} 

\affil*[1]{\orgdiv{The UM6P Vanguard Center}, \orgname{Mohammed VI Polytechnic University}, \orgaddress{\street{Lot 660 Hay Moulay Rachid}, \city{Benguerir}, \postcode{43150}, %\state{State}
\country{Morocco}}}

\affil[2]{\orgdiv{Laboratoire de Math\' ematiques Pures et Appliqu\' ees}, \orgname{Universit\'e du Littoral C\^ote d'Opale}, \orgaddress{\street{50 Rue F. Buisson}, \city{Calais Cedex}, \postcode{B.P. 699 - 62228}, %\state{State}
\country{France}}}


\abstract{In the present work, we study how to develop an efficient solver for the fast resolution of large and sparse linear systems that occur while discretizing elliptic partial differential equations using isogeometric analysis. Our new approach combines vector extrapolation methods with geometric multigrid schemes. Using polynomial-type extrapolation methods to speed up the multigrid iterations is our main focus. Several numerical tests are given to demonstrate the efficiency of these polynomial extrapolation methods in improving multigrid solvers in the context of isogeometric analysis.}




\maketitle

\section{Introduction}\label{sec1}

This paper investigates a potent strategy for applying sophisticated computational techniques to solve complicated problems. We analyze the application of geometric multigrid methods \cite{C1,C2,C3}, which are efficient resources for resolving mathematical problems in a variety of contexts. Geometric multigrid methods (GMG) have a storied history in the field of numerical simulations, particularly in the context of solving linear systems arising from partial differential equations (PDEs). Originally conceived to address simple boundary value problems prevalent in various physical applications, these methods have evolved into a powerful tool for accelerating convergence and enhancing the efficiency of iterative solvers. We incorporate these strategies with isogeometric discretization (IGA) \cite{C4,C5}, a modern technique that makes geometric modeling and numerical analysis easier \cite{C6}. Due to the fact that the matrix size increases the number of iterations required to obtain a certain accuracy, both Krylov \cite{C7} and standard iterative approaches are not optimal. In contrast, multigrid approaches are optimal, which means that their convergence rate is independent of the size mesh. By now, it is well recognized that the basic classical multigrid smoothers do not produce multigrid solvers that have robust convergence rates in the IGA discretization's spline degree \cite{C8,C9,C10,C11,C12,C13}. This convergence worsening, which is actually a common property of conventional multigrid solvers for IGA discretization matrices, is explained in detail by using symbols \cite{C13}. Several investigations have been performed to address this limitation of multigrid techniques. For example, Hofreither and Zulehner \cite{C11} presented the mass-Richardson smoother, a smoother that uses the inverse of the mass matrix as an iteration matrix. A number of additional smoothing steps are necessary for the smoother to achieve robustness in the spline degree due to boundary effects. Moreover, Donatelli et al. \cite{C8,C10,C13} provide a multi-iterative approach that integrates the multigrid method with an efficient Preconditioned Conjugate Gradient (PCG) or Preconditioned GMRES (PGMRES) smoother \cite{C7} with the mass matrix as a preconditioner at the finest level for the fast solution of the linear systems resulting from IGA. This multi-iterative multigrid approach has a robust convergence rate with respect to the spline degree, in addition to being optimal. They demonstrate this multi-iterative method's performance through a series of numerical experiments. However, this approach has restrictions. It is based on preconditioned Krylov subspace methods specifically, which require the inversion of the mass matrix, which is computationally expensive. 
%In addition, it is restricted to linear problems. 
Our suggestion to get over these restrictions is to create a multi-iterative approach that combines vector extrapolation methods \cite{C14,C15,C16,C17,C18,C19} with classical multigrid techniques. 
This hybrid approach has the advantage of being able to be applied in two separate contexts. First, to accelerate the convergence of standard multigrid schemes; in terms of iteration count and CPU time, this approach provides better convergence outcomes than classical multigrid methods. Secondly, it can be used as a preconditioner for solving linear systems arising from the Galerkin discretization of elliptic problems. Furthermore, the vector extrapolation techniques can be employed in the multicorrector stage of the linearization process for nonlinear PDEs.

%The advantage of this hybrid method is the speed up of the convergence of standard multigrid schemes. In terms of iteration count and CPU time, this method should provide better convergence outcomes than traditional multigrid approaches. To make our method even more efficient, it can also be used as a preconditioner for solving elliptic problems. Furthermore, the use of vector extrapolation techniques allows us to address nonlinear issues, giving a more general approach for linear and nonlinear problems.
Of particular importance, we will examine the powers of polynomial-type extrapolation methods, such as the reduced rank extrapolation method (RRE) \cite{C20,C21,C22,C23} and the minimal polynomial extrapolation method (MPE) \cite{C24}. These methods are specifically used to accelerate the convergence of fixed-point iterative techniques for both linear and nonlinear systems of equations. We will introduce the ideas behind these methods, making it easy to understand how they can speed up the convergence of multigrid techniques. As a result, we propose to use polynomial extrapolation methods coupled with multigrid schemes to develop a robust solver for general elliptic equations with variable coefficients, and in a general domain. 
Through practical examples and numerical experiments, we will demonstrate the efficiency of our hybrid solver; we observe that when the spline degree increases, the iteration numbers stay uniformly bounded. Thus, we deduce that our solver which combines multigrid with polynomial extrapolation methods is robust with respect to the spline degree of the IGA discretization.

The organization of this paper is as follows: In Section \ref{sec:multigridsolv}, we start by revisiting various multigrid methods. Next, we provide some insights into IGA, in order to demonstrate the limitations of classical multigrid methods with standard smoothers for solving linear systems resulting from IGA. We show this limitation through a numerical analysis of these multigrid schemes. Since multigrid approaches can be rewritten as fixed-point methods, we introduce in Section \ref{sec:polyextrapol} polynomial extrapolation methods like RRE and MPE and we discuss the advantages of these vector extrapolation methods. Finally, in Section \ref{sec:numericalexperei}, we perform several numerical results to illustrate the efficiency and the performance of our RRE-V-cycle and MPE-V-cycle approaches for solving elliptic problems.

\section{Multigrid methods for IGA}\label{sec:multigridsolv}
Let us consider a linear PDE with Dirichlet boundary conditions. 
\begin{equation}
\begin{array}{rcl}
Lu & = & f, \label{equat1}
\end{array}
\end{equation} 
in a bounded domain $\Omega$, which we will call the unit square for simplicity. The mesh of stepsize $h$ covers the domain $\Omega$. Additionally, assume that the isogeometric analysis approach is used to discretize the equation (\ref{equat1}). This results in a set of linear equations: 
\begin{equation}
\begin{array}{rcl}
A^h u^h & = & b^h, \label{equat2}
\end{array}
\end{equation} 
where $A^h$ is a square nonsingular SPD matrix; it is called the stiffness matrix, $b^h$ is a specified vector that originates from the right-hand side of (\ref{equat1}) and the boundary conditions, and $u^h$ is the vector of unknowns. Let $e^h = u^h- v^h$ be the associated algebraic error, and let $v^h$ be the current approximation of $u^h$ provided by the multigrid iteration. The error is known to satisfy the following residual equation (\ref{equat3}): 
\begin{equation}
\begin{array}{rcl}
A^h e^h & = & r^h, \label{equat3}
\end{array}
\end{equation} 
where $r^h= b^h-A^h v^ h$ is the residual vector. Knowing that the grid $\Omega^h$ of mesh size $h$ is called the ﬁne grid, we deﬁne a coarse grid $\Omega^H$ with mesh size
$2h$. Using multigrid techniques, one may approximate the error on the coarse grid. Therefore, we must solve the coarse grid problem $A^{2h} e^{2h} = r^{2h}$, project the solution onto the fine grid, and update the previous approximation $v^h$. Then $v^h + e^h$
could be a better approximation for the exact solution $u^h$.\\
The mechanism that transforms the data from the coarse grid $\Omega^{2h}$ to the fine grid $\Omega^{h}$ is called the interpolation or prolongation operator $P_{2h} ^{h}$, and the one who transforms the data from the fine grid to the coarse grid is called the restriction operator $R_h ^{2h}$ (which is often chosen as $(P_{2h}^h )^{T}$). These two operators are defined as follows:

$$ P_{2h}^h: \Omega^{2h} \longrightarrow \Omega^{h},$$
and
$$ R_h ^{2h}: \Omega^{h} \longrightarrow \Omega^{2h}.$$
In the multigrid methods, we need to define the coarse grid matrix $A^{2h}$, we can take $A^{2h}$ as the result of the isogeometric discretization to the differential operator on $\Omega^{2h}$, i.e., the $\Omega^{2h}$ version of the initial matrix system's $A^h$. We can also define $A^{2h}$ by the Galerkin projection (it is also called the Galerkin condition): Since the residual equation on the fine grid is given by (\ref{equat3}), we suppose that the error $e^h$ is in the range of the interpolation operator; therefore, there is a vector $e^{2h}$ on $\Omega^{2h}$ such that $e^h= P_{2h}^h e^{2h}$. Hence,
$$A^ h P_{2h}^h e^{2h}= r^h, $$
when we multiply the two sides of this equation by the restriction operator $R_{h}^{2h}$, we obtain
$$  R_{h}^{2h} A^ h P_{2h}^h e^{2h}= R_{h}^ {2h} r^h. $$
If we compare this last equation with $A^{2h} e^{2h} = r^{2h},$ we conclude that the coarse grid matrix is given by (\ref{equat4}):
\begin{equation}
\begin{array}{rcl}
A^{2h} & = & R_{h}^{2h} A^ h P_{2h}^h. \label{equat4}
\end{array}
\end{equation} 

\subsection{Multigrid algorithms}\label{subsubsec1}
We begin with a review of the various classical multigrid algorithms: the Two-Grid method (TG), followed by a larger class of multigrid cycling systems called the $\mu$-cycle approach. We note that in practical applications, only the values of $\mu=1$ resulting in the V-cycle and $\mu=2$ yielding the W-cycle are typically utilized.

First, we introduce the Two-Grid algorithm, a fast iterative approach to solving linear equations resulting from the discretization of a linear PDE. The suggested Two-Grid scheme is as follows, refer to Algorithm~\ref{algo1}.

\begin{algorithm}
\caption{Two-Grid algorithm}\label{algo1}
\begin{algorithmic}[1]
\Require $ A^{h}$, $ b^{h}$, $ v^{h}$: a starting point, $ \nu_1$: the number of relaxations in the pre-smoothing step, $ \nu_2$: the number of relaxations in the post-smoothing step, $ \text{Itermax}$
\Ensure $u^{h}: \text{the approximate solution of}\, A^{h} u^h = b^h $
\State $k \Leftarrow 0$
\While{$k \leq \text{Itermax and not converged}$}
            \State $u^h \Leftarrow  \text{smoother}(A^h,b^h,v^h,\nu_1)$ \Comment{Apply a relaxation method}
            \State $r^h \Leftarrow  b^h- A^h u^h$ \Comment{Compute the residual}
            
            \State $r^{2h}\Leftarrow  R_h^{2h} r^h$ \Comment{Restrict the residual}
            \State $e^{2h} \Leftarrow  \text{smoother}(A^{2h},r^{2h}, e^{2h},\nu_1)$  \Comment{Solve the coarse equation}
            \State $u^h \Leftarrow  u^h + P_{2h}^{h}e^{2h}$  \Comment{Correct the approximation}
            \State $ u^h \Leftarrow  \text{smoother} (A^h,b^h,u^h,\nu_2)$ \Comment{Apply a relaxation method}
            \State $k \Leftarrow  k + 1$
\EndWhile
\end{algorithmic}
\end{algorithm}

Then, we present the $\mu$-cycle algorithm, see Algorithm~\ref{algo2}. It performs many correction iterations at each level of the grid hierarchy, offering additional improvement in convergence for iterative solution methods.

\begin{algorithm}
\caption{$\mu$-cycle algorithm }\label{algo2}
\begin{algorithmic}[1]
\Require $A^{h}, b^{h}, v^{h}, \nu_1, \nu_2, \mu, \text{Itermax}$ 
\Ensure $u^{h}$
\State $k \Leftarrow 0$
\While{$k \leq \text{Itermax and not converged}$}
 \If {$\Omega^h$ is the coarsest grid}
            \State $u^h \Leftarrow (A^h)^{-1} b^h$   \Comment{Direct solve on the coarsest grid}
           
        \Else

            \State $u^h \Leftarrow  \text{smoother}(A^h,b^h,v^h,\nu_1)$ \Comment{Pre-smooth using relaxation}
            \State $r^h \Leftarrow  b^h- A^h u^h$  \Comment{Compute the residual}
            \State $r^{2h}\Leftarrow  R_h^{2h} r^h$  \Comment{Restrict the residual}
            \State $e^{2h} \Leftarrow 0$ \Comment{Initialize coarse grid error}
            \State $e^{2h} \gets \text{MG$\mu$} (A^{2h},r^{2h},e^{2h},\nu_1,\nu_2)$ $\mu$ times  \, \Comment{$\mu$-cycle recursion on coarse grid}
            \State $u^h \Leftarrow  u^h + P_{2h}^{h}e^{2h}$  \Comment{Correct the approximation}
            \State $ u^h \Leftarrow  \text{smoother} (A^h,b^h,u^h,\nu_2)$   \Comment{Post-smooth using relaxation}
\EndIf
\EndWhile
\end{algorithmic}
\end{algorithm}

\subsection{Multigrid: A fixed-point analysis}
We will conduct an analysis of multigrid methods to demonstrate that they can be reformulated as fixed-point methods. Therefore, Proposition~{\upshape\ref{prop1}} gives the iteration matrix of the Two-Grid method.

\begin{proposition}  \label{prop1} 
We denote by  $S_{sm}$ the iteration matrix of the smoother. Hence, the iteration matrix of the Two-Grid method is given by:
 $$ B := S_{2lev}= S_{sm}^{\nu_2}\left ( I-P_{2h}^{h}(A^{2h})^{-1} R_{h}^{2h}A^{h} \right )S_{sm}^{\nu_1},$$
and the Two-Grid sequence is as follows:
 $$u_{k+1}= B u_{k}+c,$$
 where $c = S_{\text{sm}}^{\nu_2} P_{2h}^{h} (A^{2h})^{-1} R_{h}^{2h} b^h.$
\end{proposition}

\begin{proof}
    See for instance \cite{C25}.
\end{proof}
The levels of the multigrid hierarchy are numbered by $0, . . . , l$, where level $0$ is the coarsest grid. The iteration matrix of the two-level method on level $l$ is denoted by $B_l$ and it has the form 
$$ B_l= S_{sm,l}^{\nu_2}\left ( I-P_{l-1}^{l}(A_{l-1})^{-1} R_{l}^{l-1}A_{l} \right )S_{sm,l}^{\nu_1}.$$
Thus, in the following Proposition~{\upshape\ref{prop2}}, we present the iteration matrix of the multigrid schemes.
\begin{proposition}[\cite{C25}
]\label{prop2}
The iteration matrix for the multigrid method is:
\begin{align*}
B_{mg,l} & = B_l,\;\;\;\;\;\;\;\;\;\;\;\;\; \textit{if}\; l=1,&\\
B_{mg,l}   & = S_{sm,l}^{\nu_2}\left [ I-P_{l-1}^{l}(I-B_{mg,l-1}^{\mu})A_{l-1}^{-1} R_{l}^{l-1}A_{l}\right] S_{sm,l}^{\nu_1}  &\\
 & = B_l+ S_{sm,l}^{\nu_2}P_{l-1}^{l}B_{mg,l-1}^{\mu}A_{l-1}^{-1} R_{l}^{l-1}A_{l}S_{sm,l}^{\nu_1}, \;\;\;\;\;\; \textit{if}\; l>1. &
\end{align*}

The multigrid sequence is defined recursively by:
$$ u_{k+1}= B_{mg,l} u_{k} + S_{sm,l}^{\nu_2} P_{l-1}^{l}  (I - B_{mg,l-1}^{\mu}) A_{l-1}^{-1} R_{l}^{l-1}b.$$

\end{proposition}

Next, we give some details of isogeometric analysis and recall the important properties of spline basis functions “B-splines”
 (Proposition~{\upshape\ref{prop3}}).

\subsection{B-splines $\&$ IGA}
In one-dimensional space, a vector of nodes, or “Knots vector,” $\xi_1,\xi_2,...,\xi_{N+p+1}$ is a non-decreasing
set of coordinates in the parameter space, where $N$ is the number of control points and $p$ is
the degree of the spline. If the nodes $\xi_{i}, \; i=1,..., N+p+1$ are equidistant, we say that this
vector of nodes is uniform. If the nodes in the first and last positions are repeated $p+1$
times, we say that this vector of nodes is an “open Knots vector.”.

Given $m$ and $p$ as natural numbers, let us consider a sequence of non-decreasing real numbers: $T= \{t_{i}\}_{0\leq i \leq m}$. T is called the knots sequence. From a knots sequence, we can generate a B-Spline family using the recurrence formula (\ref{moneq}). For more details, see \cite{C26,C27,C28}.

\begin{definition}[B-Splines using Cox-DeBoor Formula]
 The jth B-spline of degree $p$ is defined by the recurrence relation:
 \begin{equation}
    N^p_j = \frac{t-t_j}{t_{j+p}-t_j}  N^{p-1}_j + \frac{t_{j+p+1}-t}{t_{j+p+1}-t_{j+1}}  N^{p-1}_{j+1},
 \label{moneq}
  \end{equation}   
       where 
       $$ N^0_j(t)= \chi_{[t_j,t_{j+1}]}(t),$$
for $0<j<m-p-1.$           
 \end{definition}
%\subsubsection*{$\quad$ B-splines properties and examples} 
\begin{proposition}[B-Splines properties]\label{prop3}
\begin{itemize}
\item Compact support. $N^p_j(t)=0 \, \forall t \notin[t_j,t_{j+p+1}).$ 
 \item If $t \in[t_j,t_{j+1}),$ then only the B-splines $\{ N^p_{j-p},..., N^p_j\} $  are non vanishing at t.  
\item Non-negativity. $ N^p_j(t)\geq 0 \,\,\forall t \in [t_j,t_{j+p+1}).$ 
 \item Partition of the unity. $\sum{N^p_i(t)}=1,\, \forall t \in \mathbb{R}.$ 
\end{itemize}      
\end{proposition}

\begin{proof}
    For the proof of these properties, see \cite{C4}.
\end{proof}

Examples of B-spline basis functions with different degrees $p$ are given in Figure \ref{bsplines examples}.
  \begin{figure}[h]
      \centering
      \includegraphics[width=5cm,height=5cm]{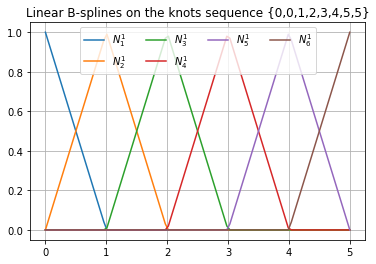}
    \includegraphics[width=5cm,height=5cm]{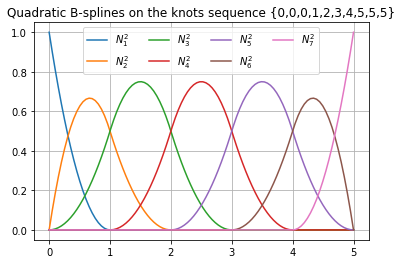}
     \includegraphics[width=5cm,height=5cm]{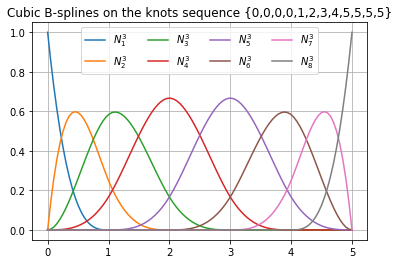}
      \caption{B-spline basis functions of order $p=1,2,3$}
      \label{bsplines examples}
  \end{figure}

\subsection{Investigating the limitations of standard multigrid techniques}

The limitations of classical multigrid methods will be discussed in this section in the context of isogeometric analysis. In particular, we will identify how p-dependence affects classical multigrid schemes used on B-spline spaces in IGA. Additionally, p-pathology's effects through symbol analysis are discussed in depth in \cite{C13,C26,C8}. Because of this, when using a standard multigrid method to solve, for high spline degree's $p$, the linear system resulting from IGA discretization, we anticipate fundamental challenges, in particular a slow (but optimal) convergence rate.
%The study will clarify the existence of numerical singularities particularly at ±$\pi$, that are characterized by infinite stiffness symbols.
We begin by demonstrating these restrictions on the $1D$ Poisson issue solutions. 
%On one hand, the graph of the symbol $f_p$ (for the definition of the symbol and its properties see \cite{C7,C11,C12,C20}), normalized by its maximum $M_{f_p}$, is displayed in Figure~\ref{pathological behavior} for $p$ values ranging from $1$ to $8$. As $p$ approaches $\infty$, the value of $f_p(\pi )/M_{f_p}$ falls exponentially to zero. Numerically speaking, we can state that for high values of spline degree $p$, the normalized symbol $f_p/M_{f_p}$ disappears at both its "mirror point" of $\theta = \pi$ and $\theta= 0$. Because of this, when using a standard multigrid method to solve, for high spline degree's $p$, the linear system resulting from IGA discretization, we anticipate fundamental challenges, in particular a slow (but optimal) convergence rate.
We will prove numerically that the convergence of the classical multigrid schemes is dependent on the spline degree. Thus, we applied first of all the V($\nu_{1}$,$\nu_{2}$)-cycle\,(where  $\nu_{1}$, $\nu_{2}$\; are the relaxation sweeps before/after the correction step) with Weighted-Jacobi \cite{C29} to solve the $1D$ Poisson problem (\ref{chj}) on fine grids with $N=16,32,64,128,256$ points. For a number of levels, $\text{nlevels}=4$ and for each spline degree $p$ a tolerance of $10^{-12}$ is set and the number of iterations (“iter”), the discrete $L_{2}$ norms of the residual (“$L_{2}\text{-res}$”) and of the error (“$L_{2}\text{-err}$”) are reported in Table \ref{table1}.

\begin{equation}\label{chj}
\left\{
\begin{array}{rrrrr}
-u_{xx} &=&f(x),&\quad \text{in}\quad \Omega,\quad \\
u&=&0,&\quad \text{on} \quad \partial \Omega,\quad \\
\end{array}
\right.
\vspace{-0.25cm}
\end{equation} \\
where\, $\Omega = \{x :0 <x< 1\}$ \,and $f(x)=  (2\pi)^2\sin(2\pi x)$. Knowing that the analytical solution to this problem is, $$u(x)=\sin(2\pi x).$$

\begin{table}[h!]
\centering
  \caption{ Convergence of V-cycle(1,1) scheme in $1D$ with nlevels=4}\label{table1}
		\begin{tabular}{| p{0.7cm}|p{1cm}|p{1.25cm}|p{1.25cm}|p{1.25cm}|p{1.25cm}|p{1.25cm}|}
			\hline
			Grid & & $p=2$& $p=3$ &  $p=4$& $p=5$ &$p=6$\\
			%m+n & 700 & 900 & 1200 & 2000& 3000\\  
			% $\alpha$ & $\alpha=0.1$& $\alpha=1$    & $\alpha=0.1$  & $\alpha=1$   \\   
			\hline
			$16$& iter $L_{2}\text{-res}$ $L_{2}\text{-err}$ &  6\hspace{0.8cm} 3.01e-14    2.18e-04  &11\hspace{0.8cm}  6.66e-14      1.60e-05 &  22\hspace{0.8cm} 2.82e-13      1.03e-06 &  39 \hspace{0.9cm} 6.86e-13 6.76e-08 & 87   
			\hspace{0.8cm}   9.16e-13 4.19e-09
			\\
			%& & & & &  	\vspace{0.1cm}
			\hline 
			$32$&iter $L_{2}\text{-res}$ $L_{2}\text{-err}$ &   6\hspace{0.8cm} \hspace{0.1cm}   9.26e-14    2.61e-05  &10\hspace{0.8cm}   5.95e-13  9.49e-07 &20\hspace{0.8cm}  7.38e-13    3.02e-08 &40\hspace{0.8cm} 7.09e-13      9.64e-10& 81\hspace{0.8cm}  9.91e-13  2.93e-11         			\\
			\hline
			$64$&iter $L_{2}\text{-res}$ $L_{2}\text{-err}$ &   6\hspace{0.8cm} \hspace{0.1cm}   4.00e-14       3.23e-06 
   &10\hspace{0.8cm}  5.27e-13 5.85e-08 &20\hspace{0.8cm}  5.89e-13        9.31e-10  &40\hspace{0.8cm} 6.14e-13        1.46e-11 & 81 \hspace{0.8cm} 7.41e-13   2.26e-13      \\
			%& & & & &  	\vspace{0.1cm}
			\hline  
			$128$&iter $L_{2}\text{-res}$ $L_{2}\text{-err}$ &   6\hspace{0.8cm} \hspace{0.1cm}  1.04e-14     4.02e-07 &10\hspace{0.8cm}  3.09e-13   3.64e-09 &20\hspace{0.8cm}  3.72e-13     2.90e-11 &39\hspace{0.8cm} 8.65e-13        2.28e-13& 79\hspace{0.8cm}   7.9e-13   3.5e-14      
			\\
            \hline 
            $256$&iter $L_{2}\text{-res}$ $L_{2}\text{-err}$ &  8\hspace{0.8cm} \hspace{0.1cm}  8.16e-13     5.03e-08 &13\hspace{0.8cm} 3.91e-13 2.27e-10 &19\hspace{0.8cm} 9.81e-13       9.18e-13 &38\hspace{0.8cm}   9.45e-13       5.38e-14& 78\hspace{0.8cm}    7.65e-13      
      2.45e-14    
			\\
           \hline 
		\end{tabular}
\end{table}

In Figure \ref{figure2} and Figure \ref{figureWcycle}, we have plotted the $L_2$ norms of the residual of the V-cycle and the W-cycle for spline degrees ranging from 2 to 7. These experiments were conducted at different levels of discretization, “nlevels”  to examine how the residual behaves as the spline degree varies. Notably, we observed that as the degree of B-spline basis functions increased, the number of iterations required to achieve convergence also increased. This phenomenon suggests a deterioration in the convergence behavior of the multigrid method with higher spline degrees \cite{C8,C11,C9,C10,C30}. Interestingly, this trend held consistent even when altering the number of discretization levels. These findings underscore the sensitivity of the multigrid method's convergence to the complexity of the basis functions, highlighting the need for careful consideration of spline degrees in isogeometric analysis applications. We notice also that we have the same results when we use the other standard relaxation schemes, Jacobi and Gauss-Seidel, as smoothers in the multigrid methods (V-cycle or W-cycle).

\begin{figure}[h]
\centering
    \includegraphics[width=5cm,height=5cm]{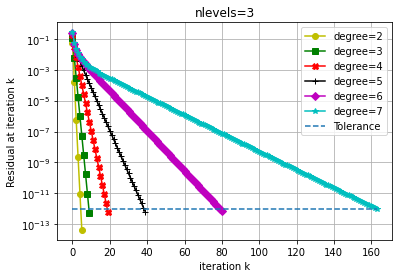}
    \includegraphics[width=5cm,height=5cm]{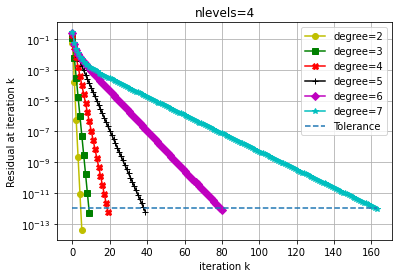}
    \includegraphics[width=5cm,height=5cm]{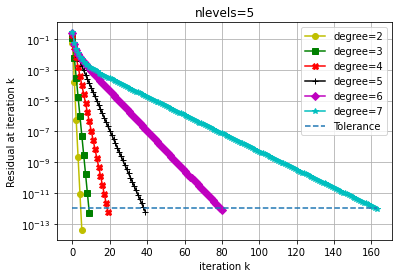}
   % \captionsetup{justification=centering}
    \caption{Performance of the V-cycle scheme in $1D$ with $N=64$}
    \label{figure2}
\end{figure}

\begin{figure}[ht!]
\centering
    \includegraphics[width=5cm,height=5cm]{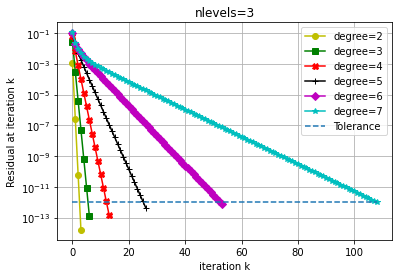}
    \includegraphics[width=5cm,height=5cm]{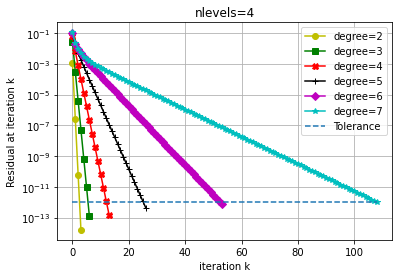}
    \includegraphics[width=5cm,height=5cm]{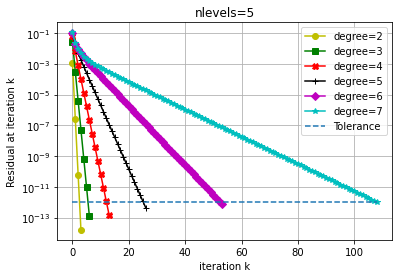}
    %\captionsetup{justification=centering}
    \caption{Performance of the W-cycle scheme in $1D$ with $N=64$}
    \label{figureWcycle}
\end{figure}

\subsection{Multigrid iterations}
For purposes of simplicity, the sequence produced by the multigrid methods (refer to Algorithm~\ref{algo1} and Algorithm~\ref{algo2}) will be noted $(s_{k})_k$. Therefore, multigrid methods can be written as a fixed-point iteration method:
\begin{equation}\label{eq:fixedpoint}
s_{k+1}= G(s_{k}),
\end{equation}
where $s_{0}$ is a starting point and $G$ (which depends on the parameters of the method) is defined in Proposition~{\upshape\ref{prop1} for the Two-Grid scheme and in Proposition~{\upshape\ref{prop2} for the multigrid method.

%Using Proposition~{\upshape\ref{prop2}, we can define the iterations $u^{(k)}$ of the multigrid methods as follows:
%\begin{equation}
  %u^{(k+1)}= B_{mg,l} u^{(k)}+b,  
 % \label{eq:mgfixedpoint}
%\end{equation}
%where $u^{(0)}$ is a starting point and $B_{mg,l} $ is the iteration matrix of the multigrid methods described in Proposition~{\upshape\ref{prop2}.

%Let $u^{(0)}$ be a starting point. Assume that $B_{mg,l} $ is the iteration matrix of the multigrid methods described in Proposition~{\upshape\ref{prop2} and that the vectors $u^{(k)}$ are produced by an iterative procedure (\ref{eq:mgfixedpoint})

%we create the new sequence $t_{0,k}$ by using the MPE and RRE algorithms described before on this sequence $(u^{(k)})_{k}$.\\

\section{Vector extrapolation methods}   \label{sec:polyextrapol}
In this section, we will first review various vector polynomial extrapolation methods to apply them to geometric multigrid methods within the context of IGA.

Let $(s_k)_{k\in \mathbb{N}}$ be a sequence of vectors of $\mathbb{R}^{\mathbb{N}}$, produced by a linear or nonlinear fixed point iteration, starting with a vector $s_0$.

\begin{equation*}
s_{k+1}= G(s_{k}),\; k=0,1,2,...
\end{equation*}
which solution is denoted by $s$. We consider the transformation $T_k$ defined by:
\[
\setlength\arraycolsep{0pt}
T_k\colon \begin{array}[t]{ >{\displaystyle}r >{{}}c<{{}}  >{\displaystyle}l } 
         \mathbb{R}^{\mathbb{N}}  &\to& \mathbb{R}^{\mathbb{N}}\\ 
          s_k &\mapsto& t_{k,q},
         \end{array}
\]
with
\begin{equation}
   t_{k,q}= \sum_{j=0}^{q} \gamma_{j}^{(q)}s_{k+j}, 
   \label{eq:newsequence}
\end{equation}
$t_{k,q}$ is the approximation produced by the reduced rank or by the minimal polynomial extrapolation methods of the limit or the anti-limit of $s_k$ as $k\rightarrow \infty$, where 
$$ \sum_{j=0}^{q} \gamma_{j}^{(q)}=1 \;\; \text{and}\; \sum_{j=0}^{q}\eta_{i,j} \gamma_{j}^{(q)}=0,\; i=0,...,q-1,$$
with the scalars $\eta_{i,j}$ defined by:
$$\eta_{i,j}= \begin{cases}
(\Delta^{2}s_{k+i},\Delta ^{2}s_{k+j}),& \text{for the RRE},   \\
 (\Delta s_{k+i},\Delta s_{k+j}),& \text{for the MPE.}  \\
  \end{cases}$$
$\Delta s_{k}$ and $\Delta^{2}s_{k}$ denote, respectively, the first and second forward differences of $s_k$ and are defined by:

\begin{equation*}
\begin{array}{rcl}
\Delta s_{k} & = & s_{k+1}- s{_k},\;\;\;\;\;\;\; \;\;k=0,1,...\\
\Delta ^ 2 s_{k}& = &  \Delta s_{k+1}-\Delta s_{k} ,\;\;\;k=0,1,...
\end{array}
\end{equation*}
Let us introduce the matrices:
$$ \Delta ^{i}S_{k,q}=[\Delta ^{i}s_{k},...,\Delta ^{i}s_{k+q-1}], \,i=1,2.$$
Using Schur’s formula, the approximation $t_{k,q}$ can be written in matrix form as 
\begin{equation}
t_{k,q}= s_{k}-\Delta S_{k,q}\left({ Y_{q}}^{T} \Delta^{2} S_{k,q} \right)^{-1} { Y_{q}}^{T}\Delta s_{k},
    \label{eq:extrapolapprox}
\end{equation}
where 
$$Y_q= \begin{cases}
\Delta^2 S_{k,q},& \text{for the RRE},   \\
 \Delta S_{k,q},& \text{for the MPE.}  \\
  \end{cases}$$
In particular, for the RRE, the extrapolated approximation $t_{k,q}$ is given by:
\begin{equation}
    t_{k,q}= s_{k}-\Delta S_{k,q}\Delta^{2} S_{k,q}^{+} \Delta s_{k},
    \label{eq:rreapproxim}
\end{equation}
where $\Delta^{2} S_{k,q}^{+}$ is the Moore-Penrose generalized inverse of $\Delta^{2} S_{k,q}$ defined by:
$$  \Delta^{2} S_{k,q}^{+}=\left({\Delta^{2} S_{k,q}}^{T} \Delta^{2} S_{k,q} \right)^{-1} {\Delta^{2} S_{k,q}}^{T}. $$
Let $\tilde{T_k}$ be the new transformation from $T_k$ given by: 
\[
\setlength\arraycolsep{0pt}
\tilde{T_k}\colon \begin{array}[t]{ >{\displaystyle}r >{{}}c<{{}}  >{\displaystyle}l } 
         \mathbb{R}^{\mathbb{N}}  &\to& \mathbb{R}^{\mathbb{N}}\\ 
          s_k &\mapsto& \tilde{t}_{k,q},
         \end{array}
\]
with 
$$\tilde{t}_{k,q}= \sum_{j=0}^{q} \gamma_{j}^{(q)}s_{k+j+1},$$
$\tilde{t}_{k,q}$ is a new approximation of the limit or the anti-limit of $s_k, k\rightarrow \infty$.
Therefore, the generalized residual of $t_{k,q}$ has been defined as follows:

\begin{equation}
\begin{array}{rcl}
\tilde{r}(t_{k,q}) & = & \tilde{t}_{k,q}- t_{k,q} \\
& = &  \Delta s_{k}-\Delta^{2} S_{k,q}\left({ Y_{q}}^{T} \Delta^{2} S_{k,q} \right)^{-1} { Y_{q}}^{T}\Delta s_{k}.
\end{array}
\label{eq:resgeneralized}
\end{equation} 

\subsection{Polynomial extrapolation algorithms}
An effective and reliable implementation for the reduced rank and the minimal polynomial extrapolation methods was provided by \cite{C31}. See \cite{C32,C33} to examine other recursive methods suggested in the literature for implementing these polynomial vector extrapolation strategies.\\
In the following, Algorithm~\ref{algo3} describes the RRE method \cite{C22,C23,C34}, and the MPE method \cite{C31} is implemented in full in Algorithm~\ref{algo4}.

\begin{algorithm}
\caption{The RRE method}\label{algo3}
\begin{algorithmic}[1]
\Require $\text{Vectors} \,s_{k}, s_{k+1},...,s_{k+q+1}$
\Ensure $t_{k,q}^{\text{RRE}}: \text{the RRE extrapolated  approximation} $
\State $ \text{Compute}\, \Delta s_{i}= s_{i+1}-s_{i},\,\text{for} \, i=k,k+1,...,k+q$
\State $\text{Set } \Delta S_{q+1}=[\Delta s_{k},\Delta s_{k+1},...,\Delta s_{k+q}]$
\State $\text{Compute the QR factorization of}\, \Delta S_{q+1}, \text{namely}, \,\Delta S_{q+1}=Q_{q+1}R_{q+1}$
\State $ \text{Solve the linear system}\, R_{q+1}^{T}R_{q+1}d=e, \,\text{where} \,d=[d_0,...,d_q]^{T} \,\text{and}\, e=[1,...,1]^{T}$
\State $ \text{Set}\, \lambda=(\sum_{i=0}^{q}d_{i})\, \text{and } \,\gamma_{j}= \frac{1}{\lambda}d_{i}, \,\text{for} \,i=0,...,q$
\State $ \text{Compute } \alpha=[\alpha_0,\alpha_1,...,\alpha_{q-1}]^{T}\, \text{where}\, \alpha_0=1-\gamma_0 \,\text{and } \,\alpha_{j}=\alpha_{j-1}-\gamma_{j} \,\text{for} \,j=1,...,q-1$
\State $ \text{Compute}\, t_{k,q}^{\text{RRE}}=s_k+ Q_{q}(R_{q}\alpha)$
\end{algorithmic}
\end{algorithm}

\begin{algorithm}
\caption{The MPE method}\label{algo4}
\begin{algorithmic}[1]
\Require $\text{Vectors} \,s_{k}, s_{k+1},...,s_{k+q+1}$
\Ensure $t_{k,q}^{\text{MPE}}: \text{the MPEE extrapolated  approximation} $
\State $ \text{Compute}\, \Delta s_{i}= s_{i+1}-s_{i},\,\text{for} \, i=k,k+1,...,k+q$
\State $\text{Set } \Delta S_{q+1}=[\Delta s_{k},\Delta s_{k+1},...,\Delta s_{k+q}]$
\State Compute the QR factorization of $\Delta S_{q+1}$, namely, \,$\Delta S_{q+1}=Q_{q+1}R_{q+1} (\Delta S_{q}=Q_{q}R_{q}$ is contained in \, $\Delta S_{q+1}=Q_{q+1}R_{q+1})$
\State Solve the upper triangular linear system $R_{q}d=-r_q,$ \,\text{where} \,$d=[d_0,...,d_q]^{T}$ \,\text{and}\,$ r_{q}=[r_{0q},r_{1q},...,r_{q-1,q}]^{T}$
\State $ \text{Set}\,d_q=1 \text{and compute}\, \lambda=(\sum_{i=0}^{q}d_{i})\, \text{and set} \,\gamma_{j}= \frac{1}{\lambda}d_{i}, \,\text{for} \,i=0,...,q$
\State $ \text{Compute } \alpha=[\alpha_0,\alpha_1,...,\alpha_{q-1}]^{T}\, \text{where}\, \alpha_0=1-\gamma_0 \,\text{and } \,\alpha_{j}=\alpha_{j-1}-\gamma_{j} \,\text{for} \,j=1,...,q-1$
\State $ \text{Compute}\, t_{k,q}^{\text{MPE}}=s_k+ Q_{q}(R_{q}\alpha)$
\end{algorithmic}
\end{algorithm}

\subsection{Restarted extrapolation methods}
The methods given in Algorithm \ref{algo3} and Algorithm \ref{algo4} become very expensive as $q$ increases. For the RRE and MPE methods in their complete forms, the number of calculations required increases quadratically with the number of iterations $q$, and the storage cost increases linearly. We find out through Sadok's work \cite{C17} that the number of operations (additions and multiplications) and the amount of storage needed to compute the approximation $t_{k,q}$ with RRE and MPE are, respectively, $2Nq^{2},$ and $q+1$, where $N$ is the dimension. In order to solve linear systems of equations, $q+1$ Mat-Vec products with RRE and MPE are needed. A good way to keep storage and computation costs as low as possible is to restart these algorithms periodically. We denote the restarted versions of the extrapolation techniques by RRE($q$) and MPE($q$). Algorithm \ref{algo5} describes the restarted method.

\begin{algorithm}
\caption{Restarted Extrapolation Algorithm}\label{algo5}
\begin{algorithmic}[1]
\Require $q$: the restart number, $s_0$: an initial vector, $\text{Itermax}$: the maximum number of iterations for the restarted algorithm to converge
\Ensure $t_{q}: \text{the extrapolated  approximation} $
\State $k \Leftarrow 0$
\While{$k \leq \text{Itermax and not converged}$}
\While{$i \leq q$}
 \State $ \text{Compute }\, s_{i+1}  \text{by} \,(\ref{eq:fixedpoint})$
 \State $i \Leftarrow i + 1$
\EndWhile
\State  $\text{Calculate}\,t_{q}\, \text{using one of the desired Algorithm}\, \ref{algo3}\; \text{or Algorithm}\,\ref{algo4}$
\If{$t_q\, \text{satisfies accuracy test}$}
        \State $\textbf{return}\, t_q $
\Else
        \State $s_0 \Leftarrow t_q$
        \State $k \Leftarrow k+ 1$
\EndIf
\EndWhile

\end{algorithmic}
\end{algorithm}

\section{Numerical experiments}   \label{sec:numericalexperei}
In the next sections, we are interested in the application of polynomial extrapolation methods like RRE and MPE as convergence accelerators on the multigrid method to solve the Poisson problem with isogeometric discretization in one (\ref{chj}) and two (\ref{jdd}) dimensions.

\subsection{The $1D$ Poisson problem}
%\subsection*{$\quad$\rom{4}.1 Applications of the extrapolation methods for accelerating multigrid iterations}
 Now, we use the V-cycle(1,1) with Weighted-Jacobi as a smoother with $\omega=2/3$ to solve the $1D$ Poisson problem (\ref{chj}) with IGA, and we show the performance of the RRE and MPE methods as an accelerator of the V-cycle iterations. The methods that follow are RRE($q$)-V-cycle and MPE($q$)-V-cycle, where $q$ is the restart number. We implemented our codes in a sequential manner. The stopping criterion is the residual with a tolerance of $10^{-12}$, and the initial guess is the zero vector. We mention that in all numerical results, “iter” indicates the number of iterations covered by the V-cycle method and the number of cycles covered by the restarted RRE and MPE methods with $q$ iterations of the V-cycle. However, the overall number of iterations in the restarted MPE and RRE methods, including the $q$ multigrid iterations within each cycle, is referred to as the “Global Iteration Count.”.

Figure \ref{figureres} and Figure \ref{fig3} display plots of the Euclidean norm of the residual and the $L_2$ norm of the error, both as functions of the global iteration number for the V-cycle, the RRE($q$)-V-cycle, and MPE($q$)-V-cycle methods. It's worth noting that in these experiments, $N=64$ was fixed for different spline degrees $p=6$, $p=7$ and $p=8$. The results demonstrate that the RRE(4)-V cycle and MPE(4)-V-cycle approaches have higher convergence rates than the V-cycle scheme, and the convergence results of both extrapolation techniques are comparable. Furthermore, we observe that convergence is enhanced in restarted extrapolation techniques by increasing the number of steps $q$.

\begin{figure}[ht!]
\centering
\includegraphics[width=5cm,height=5cm]{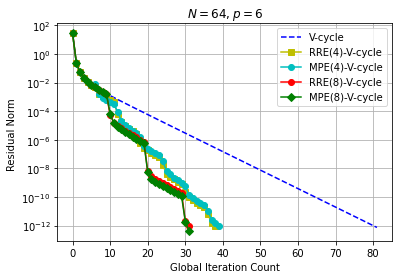}
\includegraphics[width=5cm,height=5cm]{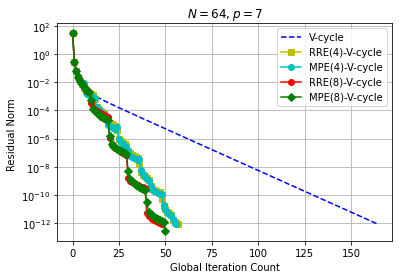}
    \includegraphics[width=5cm,height=5cm]{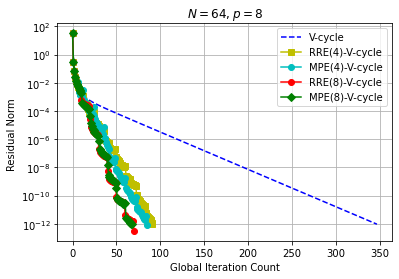}
  \caption{The norms of the residual of the $1D$ Poisson problem as a function of the iteration number}
 \label{figureres}
\end{figure}

\begin{figure}[ht!]
\centering
    %\includegraphics[width=6.2cm,height=6.2cm]{extrapolation+V-cycle_residual.png}
    %\includegraphics[width=6.2cm,height=6.2cm]{extrapolation+V-cycle_error.png}
    %\captionsetup{justification=centering}
   
    \includegraphics[width=5cm,height=5cm]{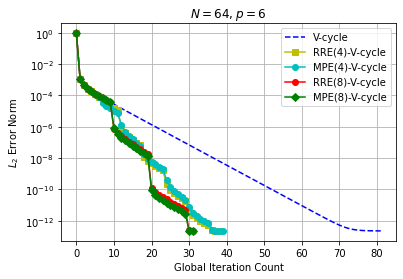}
    \includegraphics[width=5cm,height=5cm]{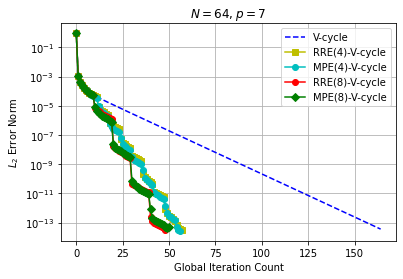}
    \includegraphics[width=5cm,height=5cm]{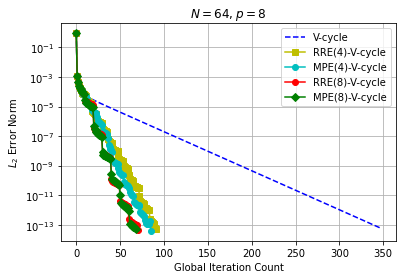}
    \caption{Error evolution of the $1D$ Poisson problem as a function of the iteration number}
    \label{fig3}
\end{figure}

This time, in the following Figure \ref{figure4} and Figure \ref{figurebehavq}, we fix the spline degree $p=7$ and we vary the number of steps $q$ (V-cycle iterations) in the restarted extrapolation methods, and we compare at each time the cycle number for RRE and MPE methods. It can be seen that the number of cycles in RRE and MPE decreases as we increase the V-cycle iterations $q$. Regarding the two polynomial extrapolation techniques, this behavior is comparable, see Figure \ref{figure4} and Figure \ref{figurebehavq}.

\begin{figure}[ht!]
    \centering
    \includegraphics[width=6.2cm,height=6.2cm]{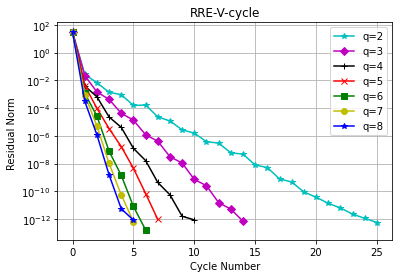}
    \includegraphics[width=6.2cm,height=6.2cm]{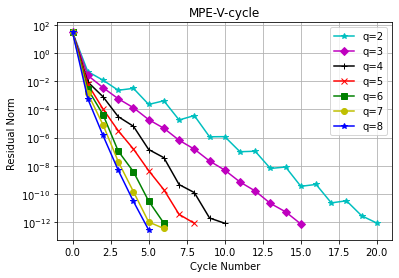}
   \caption{ Convergence results of restarted RRE and MPE combined with V-cycle for different choices of the restart number $q$. $p=7$, $N=64$}
   \label{figure4}
\end{figure}

\begin{figure}[ht!]
    \centering
    \includegraphics[width=6.2cm,height=6.2cm]{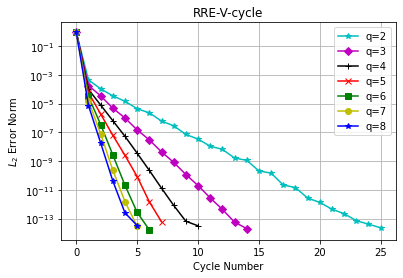}
    \includegraphics[width=6.2cm,height=6.2cm]{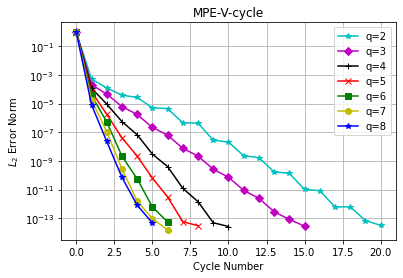}
   \caption{ $L_2$ error evolution of restarted RRE and MPE combined with V-cycle}
   \label{figurebehavq}
\end{figure}

In the following Table \ref{table4}, we aim to conduct a convergence comparison among several methods: the V-cycle, the W-cycle, the V-cycle accelerated by the RRE method (RRE($q$)-V-cycle), and by the MPE method (MPE($q$)-V-cycle).

\begin{center}
\begin{table}[ht!]
	 \caption{ Convergence efficiency comparison of multigrid schemes and its combination with extrapolation methods in $1D$ with $N=64$}
  \label{table4}	
\begin{tabular}{ |c|c|c|c|c|c| } 
\hline
Spline degree & Method &iter & $L_{2}\text{-res}$  & $L_{2}\text{-err}$ & CPU(s)\\
\hline
\multirow{3}{4em}{$p=2$} & V-cycle & 6& 4.00e-14& 3.23e-06 & 0.27 \\ 

& W-cycle & 4& 1.59e-14 &3.23e-06 & 0.39\\ 
&  RRE($q=4$)-V-cycle  & 2&  4.84e-14& 3.23e-06 & 0.30 \\ 
&  MPE($q=4$)-V-cycle  & 2 & 4.81e-14& 3.23e-06 & 0.29\\ 
%&  MPE(q=8)/V-cycle } & 2& 1.47e-14& 3.65e-07 & 0.34\\ 
\hline
\multirow{3}{4em}{$p=3$} & V-cycle & 10 &  5.27e-13&  5.85e-08 & 0.74\\ 
& W-cycle &7&   1.21e-13 &5.85e-08 & 1.01 \\ 
&  RRE($q=4$)-V-cycle  & 2&  4.87e-13& 5.85e-08 & 0.68 \\ 
& MPE($q=4$)-V-cycle  & 2 & 4.81e-13& 5.85e-08 & 0.71\\ 
%& MPE (q=8)/V-cycle } & 2& 4.38e-14& 9.14e-08 & 0.53\\ 
\hline
\multirow{3}{4em}{$p=4$} & V-cycle &20 &   5.89e-13& 9.31e-10 & 2.21\\ 
& W-cycle & 14 &   1.49e-13 & 9.31e-10 & 4.15\\ 
&  RRE($q=8$)-V-cycle  & 2&  7.24e-13& 9.31e-10 & 1.79 \\ 
&  MPE($q=8$)-V-cycle  & 2& 6.86e-13& 9.31e-10 & 1.32\\ 
\hline
\multirow{3}{4em}{$p=5$} & V-cycle & 40 & 6.14e-13&  1.46e-11 & 7.13\\ 
& W-cycle & 27& 4.37e-13&  1.46e-11 & 9.28\\ 
& RRE($q=8$)-V-cycle  & 2 &9.89e-13 &1.46e-11 & 3.56\\ 
& MPE($q=8$)-V-cycle  & 3& 3.24e-13& 1.46e-11 & 3.83\\ 
\hline
\multirow{3}{4em}{$p=6$} & V-cycle & 81 &   7.48e-13&  2.26e-13 & 19.31\\ 
& W-cycle &  54&  7.45e-13 & 2.26e-13 & 27.85 \\ 
& RRE($q=8$)-V-cycle & 4& 9.21e-13 &2.26e-13 & 6.37\\ 
& MPE($q=8$)-V-cycle & 4& 4.28e-13&  2.25e-13 & 6.43 \\
\hline
\multirow{3}{4em}{$p=7$} & V-cycle & 164 &   9.20e-13& 3.61e-14& 51.10\\ 
& W-cycle & 109&   9.95e-13 & 3.90e-14 & 71.94 \\ 
& RRE($q=8$)-V-cycle & 5 & 8.38e-13  &3.33e-14 & 12.24\\ 
& MPE($q=8$)-V-cycle  & 5& 2.84e-13& 4.84e-14 & 12.72 \\
\hline
\multirow{3}{4em}{$p=8$} & V-cycle &347& 9.66e-13&  3.00e-14 & 141.41 \\ 
& W-cycle & 232& 9.20e-13 & 3.71e-14 & 199.82\\ 
& RRE($q=8$)-V-cycle & 7&  3.22e-13 &3.18e-14 & 23.86\\ 
& MPE($q=8$)-V-cycle  &  7& 9.50e-13& 3.49e-14 & 22.81\\
\hline
\end{tabular}
\end{table}
\end{center}

We conclude that the convergence of standard multigrid methods with classical smoothers is enhanced when polynomial extrapolation techniques are combined with them. Better convergence outcomes are achieved in terms of CPU time and iteration count, especially for higher spline degrees.

\subsection{The $2D$ Poisson equation model}
We consider the $2D$ Poisson problem:
\begin{equation}\label{jdd}
\left\{
\begin{array}{rrrrr}
-\text{div}\;(\mathbf{\nabla}u) &=& f,&\quad \text{in}\quad \Omega,\quad \\
u&=&0,&\quad \text{on} \quad \partial \Omega,\quad \\
\end{array}
\right.
\vspace{-0.25cm}
\end{equation} 
where $\Omega$ is the domain of our problem which is the unit square. The exact solution and the right-hand side are given by:
$$u(x,y)= \sin(2k\pi x)\sin(2k\pi y),$$
and
$$f(x, y) = 2 (2k\pi)^{2}\sin(2k\pi x)\sin(2k\pi y).$$ 
We applied the V(1,1)-cycle to the problem (\ref{jdd}) with $\omega=2/3$ on a grid with $N=64$ points in each direction, and we start with the zero vector as an initial guess. Then, we show in Figure \ref{figure7} and Figure \ref{figerrl2pois} the $L_2$ norms of the residual and the error as functions of the global iteration number of the V-cycle with and without acceleration techniques for spline degrees $p=4$ and $p=5$. Based on the observations that we have found in the $1D$ case, we notice that we use the polynomial extrapolation techniques with restart number $q=8$.

\begin{figure}[htbp]
\centering
    %\includegraphics[width=6.2cm,height=6.2cm]{extrapolation+V-cycle_residual.png}
    %\includegraphics[width=6.2cm,height=6.2cm]{extrapolation+V-cycle_error.png}
    %\captionsetup{justification=centering}
    \includegraphics[width=6.3cm,height=6.3cm]{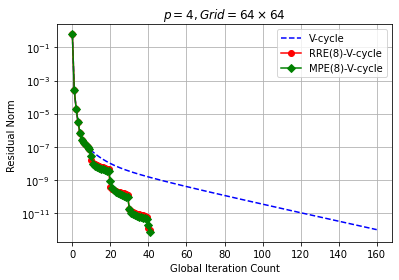}
   \includegraphics[width=6.3cm,height=6.3cm]{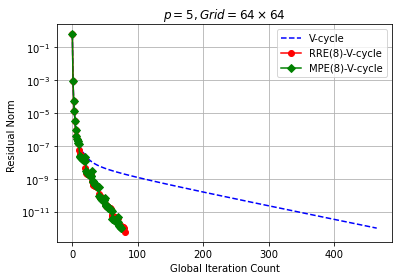}
   
    \caption{The norms of the residual of the $2D$ Poisson problem as a function of the iteration number}
    \label{figure7}
\end{figure}

\begin{figure}[htbp]
\centering
    %\includegraphics[width=6.2cm,height=6.2cm]{extrapolation+V-cycle_residual.png}
    %\includegraphics[width=6.2cm,height=6.2cm]{extrapolation+V-cycle_error.png}
    %\captionsetup{justification=centering}
     \includegraphics[width=6.3cm,height=6.3cm]{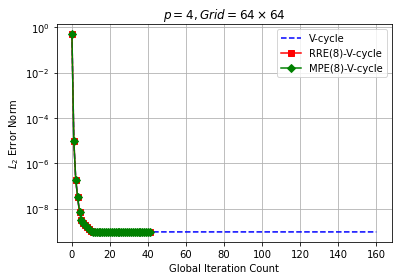}
     \includegraphics[width=6.3cm,height=6.3cm]{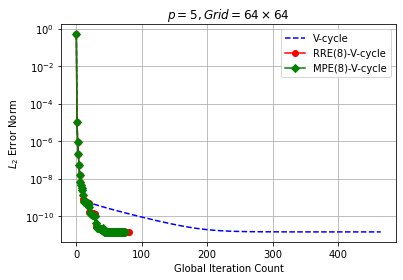}
    
    \caption{ $L_2$ error norms of the $2D$ Poisson problem as a function of the iteration number}
    \label{figerrl2pois}
\end{figure}

Similar to the $1D$ case, the findings show that the convergence rates of the RRE(8)-V-cycle and MPE(8)-V-cycle approaches are higher than those of the V-cycle scheme, and the convergence outcomes of the two approaches (RRE(8)-V-cycle and MPE(8)-V-cycle) are similar.

%We have discovered that the restarted MPE technique yields better convergence results than the restarted RRE approach based on the numerical results displayed in the $1D$ Poisson problem, so, we will concentrate on the following on the restarted MPE method.

In Table \ref{table11}, we intend to perform a comparative analysis involving a set of methods, including the V-cycle, the W-cycle, and the V-cycle enhanced by restarting the RRE and MPE process every $q$ steps. This time, our benchmark involves solving the $2D$ Poisson problem (\ref{jdd}), and we have kept the grid fixed with $N=64$ points in each direction for different values of $p$.
 
\begin{center}
\begin{table}[ht!]
\centering
  \caption{Performance of restarted extrapolation methods in comparison with classical multigrid in $2D$}
  \label{table11}
\begin{tabular}{ |c|c|c|c|c|c| } 
\hline
Spline degree & Method &  iter & $L_{2}\text{-res}$  & $L_{2}\text{-err}$ & CPU(s)\\
\hline
\multirow{3}{4em}{$p=1$} & V-cycle &10 & 3.42e-13& 4.01e-04 & 6.50  \\ 
& W-cycle & 4& 1.15e-13  & 4.01e-04 & 4.51 \\ 
&RRE(q=4)-V-cycle & 2 &1.74e-13& 4.01e-04 & 6.35 \\ 
&MPE(q=4)-V-cycle & 2 & 1.65e-13 & 4.01e-04& 6.06 \\ 
\hline
\multirow{3}{4em}{$p=2$} & V-cycle & 13 &  8.05e-13&   3.23e-06 & 12.47\\ 
& W-cycle &9 & 4.87e-13 & 3.23e-06 &  12.94\\ 
&RRE (q=4)-V-cycle &2 &  7.85e-13& 3.23e-06 & 11.07 \\ 
& MPE (q=4)-V-cycle & 2  & 7.86e-13 & 3.23e-06  & 10.89 \\ 
\hline
\multirow{3}{4em}{$p=3$} & V-cycle &49 &    9.63e-13& 5.85e-08 &94.27 \\ 
& W-cycle & 33&   8.62e-13 & 5.85e-08 & 71.84 \\ 
&RRE(q=8)-V-cycle & 3 & 7.57e-13& 5.85e-08 & 35.34 \\ 
& MPE(q=8)-V-cycle &  3 &  5.86e-13&   5.85e-08 & 36.97  \\ 
\hline
\multirow{3}{4em}{$p=4$} & V-cycle & 161 &  9.68e-13&  9.31e-10 & 506.31\\ 
& W-cycle &   108 &  9.99e-13&   9.31e-10 & 403.30 \\ 
&RRE(q=8)-V-cycle & 5 & 9.67e-13& 9.31e-10 & 120.06\\
&MPE(q=8)-V-cycle & 5 & 7.48e-13 &  9.31e-10 & 112.49 \\ 
\hline
\multirow{3}{4em}{$p=5$} & V-cycle & 466 &   9.84e-13&  1.46e-11 & 2472.61\\ 
& W-cycle & 352 & 9.84e-13  &1.46e-11  &  2110.09\\ 
& RRE(q=8)-V-cycle & 8 & 5.90e-13 & 1.46e-11 & 408.20 \\ 
&MPE(q=8)-V-cycle & 8 & 9.75e-13& 1.46e-11 & 380.39 \\

\hline
\end{tabular}
\end{table}
\end{center}

The convergence is significantly improved by combining standard multigrid approaches with polynomial extrapolation methods. Reductions in CPU time and iteration counts demonstrate this improvement in both $1D$ (Table \ref{table4}) and $2D$ (Table \ref{table11}) examples.

\subsection{The Full elliptic partial differential equation}

In order to demonstrate the efficiency of our hybrid solver, we conclude the linear case by considering the full elliptic partial differential equation (\ref{eq:fullelliptic}) with non-constant coefficients and homogeneous Dirichlet boundary conditions:
%Now, in order to demonstrate numerically the effectiveness of our extrapolation technique, we consider the full elliptic Partial
%Differential Equation (PDE) with non-constant coefficients and homogeneous Dirichlet boundary conditions:
\begin{equation}  
	\left\lbrace\begin{array}{lll}
		 -\nabla \mathbf{A}(x)\cdot \nabla\mathbf{u} + \mathbf{B}\cdot\nabla\mathbf{u} + \mathbf{c}\mathbf{u} &=~ f,&\text{ in } \Omega ,~~~~~~~~~~\\
	    \mathbf{u} &=~ 0, &\text{ on } \partial\Omega,~~~~~~~~\\
	\end{array}\right.
\label{eq:fullelliptic}
\end{equation}
where $\Omega$ is a bounded open domain in $\mathbb{R}^{d}$, $A(x): \Bar{\Omega}\longrightarrow \mathbb{R}^{d\times d}$ is a Symmetric Positive Definite (SPD) matrix of functions belonging to $\mathcal{C}^{1}(\Omega)\cap \mathcal{C}(\Bar{\Omega}),$ $B: \Bar{\Omega}\longrightarrow \mathbb{R}^{d}$ is a vector of functions in $\mathcal{C}(\Bar{\Omega})$, $c,f\in \mathcal{C}(\Bar{\Omega})$ and $c\geq 0.$

First, when $A$ is the identity matrix, $B$ is the zero vector, and $c=0$, we obtain the pure Laplacian problem. The numerical results of the pure Laplacian in $1D$ and $2D$ are already illustrated in Table \ref{table4} and Table \ref{table11}.

Next, we will demonstrate in the following that our hybrid approaches can be utilized as a preconditioner for solving linear systems resulting from the Galerkin discretization of elliptic problems, in addition to speeding up the convergence of classical multigrid methods. Therefore, we start by showing the performance of our RRE-V-cycle and MPE-V-cycle solvers for the pure Laplacian in $1D$ (\ref{chj}) and $2D$ (\ref{jdd}).
Hence, in Table \ref{tab:1dpoiss} and Table \ref{tab:2dpoiss} we vary for each grid size the spline degrees, and we illustrate the number of iterations required to reach convergence up to a precision of $10^{-12}$. We remark that the iteration numbers remain uniformly bounded as we increase the spline degree $p$. Then, we conclude that our techniques, RRE-V-cycle and MPE-V-cycle, are stable and robust with respect to the grid size and the spline degree.

\begin{table}[htbp]
\centering
\small % Utiliser \small ou \footnotesize pour les tableaux si demandé
\caption{Efficiency of extrapolation methods for the $1D$ Poisson problem with $q=8$}
\label{tab:1dpoiss}

\begin{tabular}{|r|c|c|c|c|c|c|c|}
\hline
\multicolumn{8}{|c|}{RRE(8)-V-cycle} \\ \hline
Grid & $p=2$ & $p=3$ & $p=4$ & $p=5$ & $p=6$ & $p=7$ & $p=8$ \\ \hline
32   & 1     & 1     & 1     & 2     & 3     & 3     & 6     \\ \hline
64   & 1     & 1     & 2     & 2     & 4     & 5     & 7     \\ \hline
128  & 1     & 1     & 2     & 3     & 4     & 5     & 7     \\ \hline
256  & 1     & 1     & 2     & 3     & 4     & 5     & 8     \\ \hline
\end{tabular}

\vspace{0.5cm} % Espacement entre les tableaux

\begin{tabular}{|r|c|c|c|c|c|c|c|}
\hline
\multicolumn{8}{|c|}{MPE(8)-V-cycle} \\ \hline
Grid & $p=2$ & $p=3$ & $p=4$ & $p=5$ & $p=6$ & $p=7$ & $p=8$ \\ \hline
32   & 1     & 1     & 1     & 2     & 2     & 3     & 5     \\ \hline
64   & 1     & 1     & 2     & 3     & 4     & 5     & 7     \\ \hline
128  & 1     & 1     & 2     & 3     & 4     & 5     & 7     \\ \hline
256  & 1     & 1     & 2     & 3     & 4     & 6     & 7     \\ \hline
\end{tabular}

\end{table}

\begin{table}[htbp]
\centering
\small % Vous pouvez ajuster la taille de police à footnotesize si nécessaire
\caption{Convergence results of extrapolation methods for the $2D$ Poisson problem with $q=8$}
\label{tab:2dpoiss}

\begin{tabular}{|r|c|c|c|c|c|}
\hline
\multicolumn{6}{|c|}{RRE(8)-V-cycle} \\ \hline
Grid & $p=1$ & $p=2$ & $p=3$ & $p=4$ & $p=5$ \\ \hline
$16 \times 16$   & 1 & 2 & 3 & 6 & 9 \\ \hline
$32 \times 32$   & 1 & 2 & 3 & 5 & 9 \\ \hline
$64 \times 64$   & 1 & 1 & 3 & 5 & 8 \\ \hline
$128 \times 128$ & 1 & 1 & 2 & 4 & 7 \\ \hline
\end{tabular}

\vspace{0.5cm} % Ajout d'un espace entre les deux tableaux pour une meilleure lisibilité

\begin{tabular}{|r|c|c|c|c|c|}
\hline
\multicolumn{6}{|c|}{MPE(8)-V-cycle} \\ \hline
Grid & $p=1$ & $p=2$ & $p=3$ & $p=4$ & $p=5$ \\ \hline
$16 \times 16$   & 1 & 2 & 3 & 6 & 8 \\ \hline
$32 \times 32$   & 1 & 2 & 3 & 6 & 10 \\ \hline
$64 \times 64$   & 1 & 1 & 3 & 5 & 8 \\ \hline
$128 \times 128$ & 1 & 1 & 2 & 4 & 7 \\ \hline
\end{tabular}

\end{table}

Considering this time the full elliptic PDE (\ref{eq:fullelliptic}) with non-constant and homogeneous Dirichlet boundary conditions coefficients on the unit square, we take
\begin{align*}
A(x,y) &=
\begin{bmatrix}
    (2 + \cos x)(1 + y)   & \cos(x + y) \sin(x + y) \\
   \cos(x + y) \sin(x + y) & (2 + \sin y)(1 + x)
\end{bmatrix} \\
B(x,y) &=
\begin{bmatrix}
    11+\sin(x) +y\sin(x) -2\cos(x+y) ^{2}\\
    -9-\cos(y)-x\cos(y)-2\cos(x+y)^{2}
\end{bmatrix} \\
c&=f(x,y)=1.
\end{align*}
We solve the resulting linear system of the described model problem (\ref{eq:fullelliptic}) using the RRE-V-cycle approach. We start with the zero vector as an initial guess and iterate until the $L_2$ norm of the residual is less than $10^{-12}$. The obtained results are collected in Table \ref{table13}.

 \begin{table}[ht]
        \centering
		
  \caption{Efficiency of RRE-V-cycle solver with $q=8$ for solving the full elliptic problem}
  \label{table13}
  
        \begin{tabular}{|l|l|l|l|l|l|}
\hline
 Grid  & $p=1$  & $p=2$   & $p=3$   & $p=4$ & $p=5$  \\
\hline
 $16\times 16$  & $2$  & $3$   & $4$   & $6$ & $9$    \\
 \hline
 $32\times 32$    & $3$  & $3$  &$4$    & $8$  & $9$   \\
 \hline
 $64 \times 64$  & $4$ & $4$   & $5$    & $9$   & $10$  \\
\hline
\end{tabular}
%\caption{Efficiency of MPE extrapolation method with $q=8$.}  
\end{table}

In addition, we show the convergence results, and we numerically demonstrate the efficiency of our technique on the following Advection-Diffusion problem (\ref{eq:paraboliceq}). The results are collected in the following Table \ref{table14}.

\begin{equation}
	\left\lbrace\begin{array}{lll}
		 -\mathbf{\epsilon} \Delta \mathbf{ u}+ \mathbf{B}\cdot\nabla\mathbf{u} &=~ f,&\text{ in } \Omega ,~~~~~~~~~~\\
	    \mathbf{u} &=~ 0, &\text{ on } \partial\Omega,~~~~~~~~\\
	\end{array}\right.
 \label{eq:paraboliceq}
\end{equation}
with $\epsilon= 10^{-1}$ and $B=(1,1).$

\begin{table}[h]
         \centering
      \caption{Convergence results of RRE-V-cycle solver with $q=8$}
      \label{table14}
        \begin{tabular}{|l|l|l|l|l|}
\hline
 Grid  & $p=1$   & $p=2$  & $p=3$   & $p=4$    \\
\hline
 $16 \times 16$       & $2$   & $2$   & $4$   & $6$     \\
 \hline
 $32 \times 32$       & $2$   & $2$  &$4$    & $5$     \\
 \hline
 $64 \times 64$       & $2$   & $2$   & $3$    & $6$     \\
% $128 \times 128$      & $2$   & $2$   & $3$    & $5$   \\
\hline
\end{tabular}
%\caption{Efficiency of MPE extrapolation method with $q=8$.}  
\end{table}

Now consider the full elliptic problem (\ref{eq:fullelliptic}) in $2D$ case, defined on a quarter of an annulus approximated by the B-spline approach (refer Figure \ref{figure6}).
With $\Omega=\{ (x,y)\in \mathbb{R}^{2}: r^{2} < x^{2}+y^{2} < R^{2}, x> 0, y> 0  \}, r=0.2, R=1.$ Additionally, we take

\begin{align*}
A(x,y) &=
\begin{bmatrix}
    (2 + \cos x)(1 + y)   & \cos(x + y) \sin(x + y) \\
   \cos(x + y) \sin(x + y) & (2 + \sin y)(1 + x)
\end{bmatrix} \\
B(x,y) &=
\begin{bmatrix}
-5y\\
   5x
\end{bmatrix} \\
c(x,y)&=xy,
\end{align*}
we test our solver using a manufactured solution:
%and $f(x, y)$ computed from the exact solution
$$u(x, y) = (x^{2}+y^{2}-0.2^ 2)(x^{2}+y^{2}-1) \sin x \sin y.$$

\begin{figure}[tbhp]
\centering
\includegraphics[width=6.5cm,height=6.5cm]{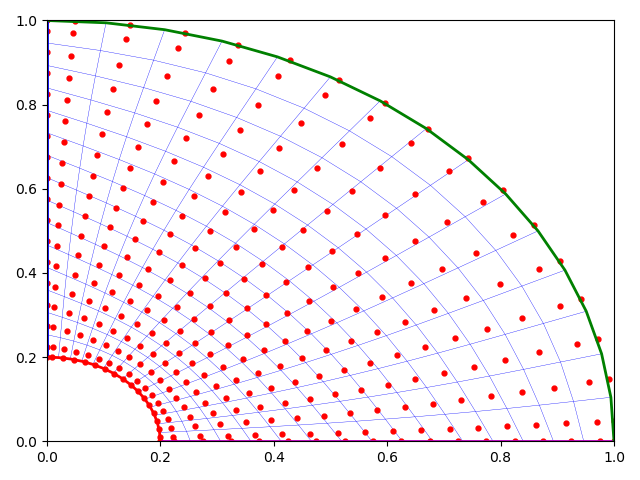}
\includegraphics[width=6.5cm,height=6.5cm]{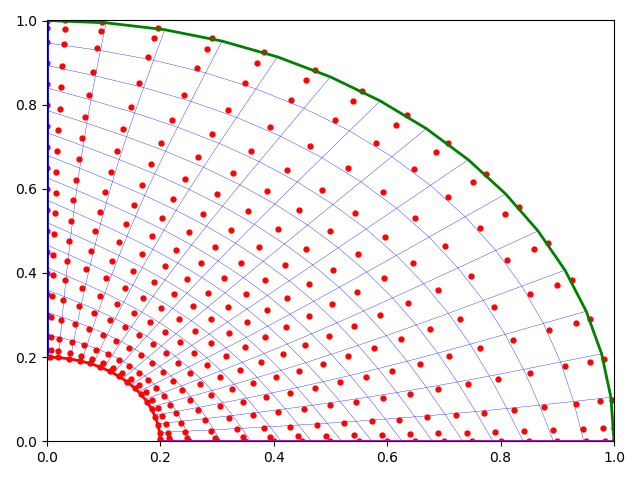}

\caption{Quarter annulus as (left) biquadratic and (right) bicubic B-spline surface}
\label{figure6}
\end{figure}

With our solver, which combines the extrapolation methods and multigrid approaches, we were able to solve the linear system that emerged from the IGA approximation of this problem. The findings are reported in Table \ref{table15}. We note that, regardless of spline degree $p$, the RRE-V-cycle exhibits a stable convergence rate.
    \begin{table}[ht]
        \centering
      \caption{Convergence results of the RRE-V-cycle solver with $q=8$, applied to the $2D$ full elliptic problem on a quarter of an annulus at different grid sizes and for different spline degrees $p$. Number of iterations to achieve the tolerance of $10^{-12}$}
      \label{table15}
        \begin{tabular}{|l|l|l|l|l|l|}
\hline
 Grid   & $p=1$   & $p=2$   & $p=3$   & $p=4$ & $p=5$  \\
\hline
 $16\times 16$       & $2$   & $2$   & $2$   & $4$  & $9$   \\
 \hline
 $32\times 32$       & $2$   & $2$   &$2$    & $4$   & $7$    \\
 \hline
 $64 \times 64$       & $2$   & $2$   & $3$    & $5$   & $6$    \\
 %$128 \times 128$      & $-$  & $$   & $$    & $$   & $$   \\
\hline
\end{tabular}
%\caption{Convergence results of the MPE extrapolation method with $q=8$, applied to the $2D$ full elliptic problem on a quarter of an annulus at different grid sizes and for different spline degrees $p$. Number of iterations to achieve the tolerance of $10^{-12}$.}  
\end{table}
\\
We sum up the numerical example by displaying the error $ \vert u(x,y)-u_{h} (x,y)\vert/ \Vert u \Vert_{\infty}$ for $N =32$ in each direction, and for spline degrees $p=2,3$ in Figure \ref{fig7}, where $u_h$ is the estimated solution. We additionally indicate in this figure the 2-norm of the relative error.

\begin{figure}[ht!]
    \centering
    
    \begin{subfigure}{0.4\textwidth}
    \includegraphics[width=6.5cm,height=6.5cm]{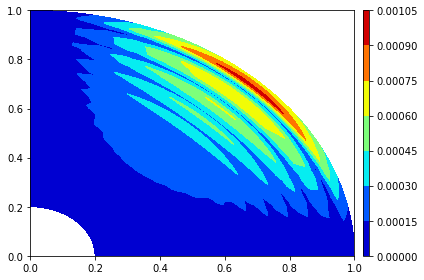}
        \caption{$N=32,p=2$, $err= 9.0 \times 10^{-4}$ }
        \label{fig:a}
    \end{subfigure}
    \hfill
    \begin{subfigure}{0.4\textwidth}
       \includegraphics[width=6.5cm,height=6.5cm]{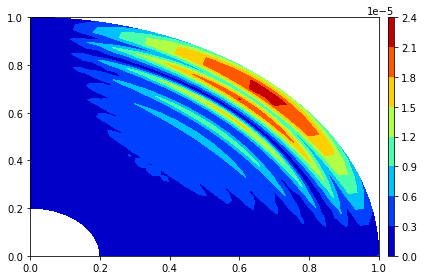}
        \caption{$N=32,p=3$, $err= 2.1 \times 10^{-5}$ }
        \label{fig:b}
    \end{subfigure}
    \caption{Error of the estimated solution for $N=32$ and for different values of spline degree $p,$ where err=$\Vert u-u_h \Vert/ \Vert u \Vert$}
   \label{fig7}
\end{figure}

%pour une certaine experirn
\newpage
\section{Conclusions}
\label{sec:conclusions}

In this paper, we have investigated the fundamental characteristics of multigrid methods in the context of isogeometric analysis. Our exploration has highlighted that these multigrid methods are optimal in the sense that their convergence rate is independent of the mesh size. However, this is not the case for the spline degree, as we have mentioned in the introduction and numerical results. Since multigrid methods can be expressed as fixed-point iterations, we have delved into the utilization of polynomial extrapolation techniques such as the reduced rank extrapolation (RRE) and the minimal polynomial extrapolation (MPE) methods. We have built a hybrid solver by combining classical geometric multigrid schemes with polynomial extrapolation methods (RRE-V cycle and MPE-V cycle), and we have found that it performs exceptionally well. The performance of our approach lies in accelerating the standard multigrid iterations. This can also be seen as a preconditioner, which provides stable and robust convergence with respect to the spline degrees. 

To illustrate the efficiency of this combined approach, we have conducted numerical experiments to solve the full elliptic problem. The consistent results demonstrate the efficiency of RRE-V-cycle and MPE-V-cycle solvers to enhance the convergence rates of multigrid iterations. As a result, alternative application of our approach based on polynomial extrapolation techniques like RRE and MPE can be extended to the multicorrector stage of the linearization for nonlinear problems.


\end{document}